\documentclass[12pt,oneside, a4paper]{amsart}
\vsize\textheight
\hsize\textwidth

\usepackage{amsfonts}
\usepackage{amsmath,amsthm,amssymb}

\usepackage{color}

\newtheorem{theorem}{Theorem}
\newtheorem{proposition}[theorem]{Proposition}
\newtheorem{lemma}[theorem]{Lemma}
\newtheorem{corollary}[theorem]{Corollary}
\newtheorem{remark}[theorem]{Remark}
\newcommand\ve{\varepsilon}
\newcommand\ff{\varphi}
\newcommand\de{\delta}
\newcommand\D{\Delta}
\newcommand\om{\omega}
\newcommand{\sign}{\mathop{\mathrm{sgn}}}
\newcommand{\NN}{\mathbb N}
\newcommand{\RR}{\mathbb R}
\newcommand{\CC}{\mathbb C}
\newcommand\PP{\mathcal{P}}
\newcommand\Q{\mathcal{Q}}
\newcommand\X{\mathcal{X}}
\newcommand\Y{\mathcal{Y}}

\newcounter{reb}
\newcounter{rev}
\newcounter{rec}
\def\EV{\lambda}
\begin{document}
\title[On the weight function in GPY for small gaps between primes]{On the optimal
weight function in the \\ Goldston-Pintz-Y\i ld\i r\i m method \\ for finding small gaps between consecutive primes}

\date{\today}

\author{B\'alint Farkas}
\address{Bergische Universit\"at Wuppertal \newline Faculty of Mathematics and Natural
  Science \newline Gau\ss stra\ss e 20, 42119 Wuppertal, GERMANY}
\email{farkas@uni-wuppertal.de}

\author{J\'anos Pintz}
\address{Alfr\' ed R\' enyi Institute of Mathematics \newline Hungarian Academy of
  Sciences \newline 1364 BUDAPEST, PO Box 127, HUNGARY}
\email{pintz@renyi.mta.hu \newline revesz.szilard@renyi.mta.hu}

\author{Szil\'ard R\'ev\'esz}
\address{Department of Mathematics \newline Kuwait University \newline P.O. Box 5969 Safat --
  13060 KUWAIT}
\email{szilard@sci.kuniv.edu.kw}


\thanks{Supported in part by the Hungarian National Foundation for Scientific Research, Project \# K-81658, K-100291, K-100461 and NK-104183.
Work done in the framework of the project ERC-AdG 228005.}

\begin{abstract}
We work out the optimization problem, initiated by
K.{} Soundararajan, for the choice of the underlying polynomial $P$
used in the construction of the weight function in the
Goldston--Pintz--Y\i ld\i r\i m method for finding small gaps between
primes. First we reformulate to a maximization  problem
on $L^2[0,1]$ for a self-adjoint operator $T$, the norm of
which is then the maximal eigenvalue of $T$. To find
eigenfunctions and eigenvalues, we derive a differential
equation which can be explicitly solved. The aimed maximal
value is $S(k)=4/(k+ck^{1/3})$, achieved by the
$k-1^{\text{st}}$ integral of
$x^{1-k/2}J_{k-2}(\alpha_{1}\sqrt{x})$, where $\alpha_1\sim c
k^{1/3}$ is the first positive root of the $k-2^{\text{nd}}$
Bessel function $J_{k-2}$. As this naturally gives rise to a
number of technical problems in the application of the GPY
method, we also construct a polynomial $P$ which is a simpler
function yet it furnishes an approximately optimal extremal
quantity, $4/(k+Ck^{1/3})$ with some other constant $C$. In the
forthcoming paper of J.{} Pintz \cite{Pintzuj} it is indeed shown how this
quasi-optimal choice of the polynomial in the weight finally
can exploit the GPY method to its theoretical limits.
\end{abstract}

\maketitle

\section{Introduction}\label{sec:intro}

\subsection{The extremal problem as given by Soundararajan.}\label{sec:Sound}
In his work \cite{S} Soundararajan presents and analyzes the
proof of Goldston--Pintz--Y\i ld\i r\i m yielding small gaps between
primes. Among others he raises and answers one of the most
important problems of the field: Is it possible to modify the
weight function $a(n)$ in such a way that the method would lead
to infinitely many bounded gaps between consecutive primes. If
we consider the weight functions in full generality, that is
all functions $a(n)$, then this leads essentially to a
tautology. For example, defining $a(n)=1$ if both $n$ and $n+2$ are
primes, and otherwise setting $a(n)=0$, the summatory function
of $a(n)$ describes the number of twin primes up to $x$. Thus
we cannot hope an asymptotic evaluation of the summatory
function. We briefly describe the feasible choices of the weight function $a$. Let us take an admissible $k$-tuple ${\mathcal H}=
\{h_1,\dots,h_k\}$ meaning that there is no
prime $p$ with the property that the elements $h_i$ of
${\mathcal H}$ cover all residue classes mod $p$. Let
$P_{\mathcal H} (n)= \prod_{i=1}^k (n+h_i)$ and let us define
$\lambda_d= \mu(d) P(\frac{\log(R/d)}{\log R})$ with a nice
function $P$, for example a polynomial, with the additional property
$\lambda_1=1$, which is equivalent to $P(1)=1$. Afterwards we
reduce our choice of $a(n)$ to those of type $a(n) =
\sum_{d\leq R, d|P_{\mathcal H}(n)} \lambda_d$ and try to evaluate
the summatory function of $a(n)$ and that of $a(n) \chi(n+h)$,
where $h$ is an arbitrary number with $h<\log n$ and $\chi$ is
the characteristic function of the primes. (In case of bounded
gaps between primes it is sufficient to consider the case when
$h=h_i, i=1,2,\dots,k$.)

Soundararajan explains, how the optimal weight function $a(n)$,
hence $\lambda_d$, should be chosen to obtain best result: see
formula (8) in \cite{S}. In order to get this optimum, he also
explains the choice $\lambda_d:=\mu(d)
P\left(\frac{\log(R/d)}{\log R}\right)$ where $P$ is some
suitably nice function, like a polynomial or at least a
sufficiently many times (at least $k$ times) differentiable,
smooth function on $[0,1]$ (or at least on $[0,1)$),
\emph{vanishing at least in the order $k$ at $0$}, and
satisfying the normalization $P(1)=1$. Then, according to the
analysis by Soundararajan, the optimal choice for $a(n)$ and
$\lambda_d$ is equivalent to looking for the maximal possible
value of (12) of \cite{S}, i.e., to determining
\begin{equation}\label{Sproblem}
S(k):=\sup_P \left(\int_0^1 \frac{x^{k-2}}{(k-2)!}
\left(P^{(k-1)}(1-x)\right)^2  dx \right) \bigg/ \left(\int_0^1
\frac{x^{k-1}}{(k-1)!} \left(P^{(k)}(1-x)\right)^2  dx \right),
\end{equation}
where the set of functions $P$, to be taken into account in the
supremum, can be the set of certain polynomials as before, or more
generally a family of functions subject to some conditions.

Soundararajan \cite{S} shows that the question whether we are
able to find in this way infinitely many bounded gaps between
primes is equivalent to the problem whether there exists any
natural number $k$ with $S(k)> 4/k$. Then he mentions that the
opposite inequality $S(k)<4/k$ holds for all $k$ and therefore
the method cannot yield infinitely many bounded prime gaps. (In
an earlier unpublished note \cite{Se} he gives the short proof
of this fact; we will reproduce this in \S
\ref{sec:estimationS(k)}. His considerations also lead easily
to the stronger inequality $S(k) < 4/(k+c\log k)$, cf. \S
\ref{sec:estimationS(k)}). Although his work answered
negatively the above mentioned central problem, it gave some
hints but did not answer the question: What is the best weight
function that can be chosen, and what size of gaps are implied
by it? In their work \cite{GPYII} Goldston, Pintz and Y\i ld\i
r\i m showed that if one takes $P(x)=x^{k+\ell}$, where $k$ and
$\ell$ are allowed to tend to infinity with the size $N$ of the
primes considered, then with several essential modifications of
the original method one can reach infinitely many prime gaps of
size essentially $\sqrt{\log p}$. (To have an idea of the
difficulties it is enough to mention that the rather condensed
proof of the result needs about 40 additional pages beyond the
original one, presented with many details and explanations in
\cite{GPYI}. However, a shortened, simplified and more
condensed version \cite{GMPY} needs only 5 pages). In this case
$\ell=c\sqrt{k}$ and the value of the fraction \eqref{Sproblem}
is $4/(k+c'\sqrt{k}$) for the given choice of
$P(x)=x^{k+\ell}$. Beyond the mentioned important fact that $k$
and $\ell$ are unbounded in \cite{GPYII}, the scheme of the
proof is similar but not the same as in the simplified version
of Soundararajan \cite{S}. However, a careful analysis suggests
that in order to find the limits of the method it is necessary
(but as discussed a little later, not necessarily sufficient)
to find the size of $S(k)$ as $k$ tends to infinity together
with the function $P$ which yields a maximum (if it exists) in
the supremum, or at least a function $P$ which yields a value
``enough close'' to the supremum.

\subsection{Conditions and normalizations}
Before proceeding, let us discuss right here the issue of
conditions and normalizations in the formulation of this
maximization problem. First, it is clear that $P^{(k)}$ remains
unchanged, if we add any constant to $P^{(k-1)}$. Thus the
extremal problem becomes unbounded under addition of a free
constant, hence at least some conditions must certainly
control this divergence.

In the number theory construction of Goldston--Pintz--Y\i ld\i
r\i m, (by now generally abbreviated as the ``GPY method'') the
natural restriction is that $P$ must be a polynomial divisible
by $x^k$---or, if we try to generalize the method, then a
$k$-times continuously differentiable function with $P^{(j)}$
vanishing at $0$ for $j=0,\dots,k-1$. That is
$P(x)=\frac{P^{(k)}(0)}{k!} x^k + o(x^{k+1})$. The reason for
that is the fact that the whole idea hinges upon the use of the
generalized M\"obius inversion, more precisely of the
$\Lambda_{j}$ function, which must be zero for numbers having
at least $k$ prime factors---always satisfied by the numbers
represented by the product form $(n+h_1)\cdots(n+h_k)$ in the
construction. So for any meaningful weight function we need to
use weights not containing any smaller power $x^j$ than $x^k$.
In other words, we should assume here $P$ having a zero of
order $k$, i.e., $P(0)=P'(0)=\cdots=P^{(k-1)}(0)=0$, while
$P^{(k)}(0)$ can be arbitrary.

The analysis of Soundararajan exposed the question, whether a
linear combination of monomials, i.e., a polynomial, or perhaps
some more sophisticated choice of a weight function, may
perhaps improve even upon this. We can say that the
\emph{theoretical limit of the GPY method} is the result,
obtainable in principle by a choice of the weight function $P$
maximizing the extremal quantity \eqref{Sproblem}. Yet it is to
be noted that the technicalities of GPY are far more
substantial than to simply ``substituting any $P$'' in it would
automatically lead to a result---it is not even that clear,
what result would follow from a given weight function.
Therefore, to test the limits of the GPY method, we should
break our approach into two parts. First, we look for the
optimization of the weight $P$, in the sense of
\eqref{Sproblem}, and second, we extend the GPY method using
that weight function. This paper is concerned with this first
question, and the second part of this program is carried out in
\cite{Pintzuj}.

The aim of the present analysis is to settle the issue of
optimization in Problem \eqref{Sproblem}. We find the optimal
order, and the maximizer of the problem \eqref{Sproblem},
furthermore, as this maximum can be achieved by a relatively
sophisticated choice of the weight function $P$---actually a
transformed Bessel function---we also construct a polynomial
weight which is approximately optimal in \eqref{Sproblem}.

Part of these results were reached by J. B.{} Conrey and his
colleagues at the American Institute of Mathematics already in
2005. Using a calculus of variation argument they found the
Bessel function $J_{k-2}$ and made some calculations for
concrete values of $k$ (without analyzing the case
$k\to\infty$). The fact that the Bessel functions may perform
better than polynomials in the GPY method is also briefly noted
in the book of J. B.{} Friedlander and H.{} Iwaniec
\cite{FriedIwa} without going into details.

\subsection{Structure of the paper}\label{sec:Structure}
In this paper we proceed along the following course.

Interpreting the problem in the widest possible function class
which makes sense (i.e., when at least the occurring integrals
exist finitely) in Section \ref{sec:maximum} we make several
further reformulations until we arrive at a maximization
problem in the Hilbert space $L^2[0,1]$. Exploiting the rich structure of Hilbert spaces, and the
particular properties of the reformulation as a certain
quadratic form with a Fredholm-type operator, we derive
existence of maximizing functions in this wide function class.
Then we also exploit the concrete form of the kernel in our
Fredholm-type operator and compute that the maximizers, or,
more generally, eigenfunctions, are necessarily smooth.
Furthermore, in Section \ref{sec:DiffEq} we find that these
eigenfunctions satisfy certain differential equations.
 The solutions are then found to be transformed variants of
certain Bessel functions. Also it turns out that the solutions
are analytic, and they yield a function value in the extremal
problem directly related to the choice of a parameter, which,
due to the initial value restriction $P^{(k-1)}(0)=0$, must be
a zero of the arising Bessel function $J_{k-2}$. Finally these
combine to the full description of the maximal value $S(k)$
together with the precise form of the extremal function. From
the well-known asymptotic formula for the first zero of the
Bessel function $J_m$, when $m\to \infty$, we derive that
$S(k)$ is precisely asymptotic to $\frac{4}{k+ck^{1/3}}$ with
a concretely known constant $c=3.7115\dots$.

Unfortunately, in spite of analyticity and power series
expansion, the found extremal function is too complicated to be
used in the number theory method of GPY. Basically, we need
restrictions on the  degree and the coefficient size for the
powers appearing in the weight function $P$ to make the
complicated method work in a technically feasible way. As
discussed in Section \ref{sec:ps}, not even calculations using
the power series expansion of Bessel functions lead to feasible
expressions. Therefore, finally we look for quasi-optimal
polynomials, which still achieve close to extremal values. The
result of the last section is the concrete construction of a
polynomial $P$ satisfying the needed technical requirements and
still achieving in \eqref{Sproblem} a ratio of the order
$\frac{4}{k+Ck^{1/3}}$ with some other constant $C$. That
suffices in the method of GPY, because the value of the
constant $C$ does not increase the order, only the arising
constants, in the final result.

Settling the issue of the search of optimal and quasi-optimal
weights, the door opens up for revisiting the method of GPY and
not only improve upon all the known results, but also push the
available techniques to the theoretical limits of that method.
This closely connected work is carried out in the paper
\cite{Pintzuj}.

\section{Reformulations and the finiteness of
$S(k)$}\label{sec:reformulation}

\subsection{Reformulations}\label{sec:reformulations} The
normalization $P(1)=1$ is rather inconvenient because the next
reformulation (still following Soundararajan) is to put
$Q(y):=P^{(k-1)}(y)$, a completely logical step in view of the
fact that no values of $P$, $P'$, etc. $P^{(k-2)}$ occur in the
actual optimization problem \eqref{Sproblem} and that the still
occurring $P^{(k-1)}$ and $P^{(k)}$ can be nicely expressed as
$Q$ and $Q'$. So in line with the restriction that $P$ vanishes
at least to the order $k$ at 0, following Soundararajan we
write
\begin{align*}
P(x)&=\int_0^x \int_0^{x_1} \dots \int_0^{x_{k-2}} P^{(k-1)} (x_{k-1})  dx_{k-1}\dots dx_1
=\int_0^x \int_0^{x_1} \dots \int_0^{x_{k-1}} P^{(k)} (x_{k})  dx_{k-1}\dots dx_1
\\& =\int_0^x \int_0^{x_1} \dots \int_0^{x_{k-2}} Q(x_{k-1})  dx_{k-1}\dots dx_1
=\int_0^x \int_0^{x_1} \dots \int_0^{x_{k-1}} Q'(x_{k})  dx_{k}\dots dx_1.
\end{align*}
Therefore, $P(0)=\cdots =P^{(k-2)}(0)=P^{(k-1)}(0)=0$
transforms to the simpler requirement that $Q(0)=0$, while the
corresponding $P$ is obtained by the above integrals directly.
Let us record one more thing here: The condition that $P(1)=1$,
expressed in terms of $Q$, is a \emph{linear restriction}, as
$I(Q):=(P(1)=)\int_0^1 \int_0^{x_1} \dots \int_0^{x_{k-2}}
Q(x_{k-1})  dx_{k-1}\dots dx_1$ is just a linear functional on
the function $Q$. To express it in a more condensed, closed
form, we may apply Fubini's theorem to get a representation in
the form of the well-known Liouville integral
\begin{align*}
P(x)&= \int_0^x Q(t) \frac{(x-t)^{k-2}}{(k-2)!} dt,\\
P(1)&= \int_0^1 Q(t) \frac{(1-t)^{k-2}}{(k-2)!} dt = \int_0^1 Q(1-y) \frac{y^{k-2}}{(k-2)!} dt.
\end{align*}
Note the similarity to the numerator of the quotient in
\eqref{Sproblem}. It is thus immediate by the Cauchy--Schwarz
inequality that $P(1)$ is a finite, convergent integral
whenever the Lebesgue integral $\int_0^1 Q^2(1-y)
\frac{y^{k-2}}{(k-2)!} dy$ exists. That is, no special
requirement is needed to this effect once we guarantee that the
numerator and denominator in \eqref{Sproblem} are well-defined.

In all, we were to look for maximum in the family
\begin{equation}\label{Pcondi}
\PP:=\left\{P(x)=\int_0^x \int_0^{x_1} \dots \int_0^{x_{k-2}} P^{(k-1)}(x_{k-1})  dx_{k-1}\dots dx_1, \quad P(1)=1, \quad P^{(k-1)}(0)=0\right\},
\end{equation}
but following Soundararajan we changed the setup to
\begin{equation}\label{Q1condi}
\Q_1:=\left\{Q~:~ ~\int_0^1 Q(1-y) \frac{y^{k-2}}{(k-2)!} dy =1, \quad Q(0)=0\right\},
\end{equation}
where now $Q$ can be any (say, continuously differentiable)
function satisfying the requirements. This also means that we
want the occurring functions to belong to a suitable function
class, to be specified later. The quantity we seek to maximize
is then expressed as
\begin{equation}\label{RatioQ}
\frac{\int_0^1 n y^{n-1} Q^2(1-y) dy}{\int_0^1 y^{n} Q'^2(1-y) dy} \qquad (n:=k-1),
\end{equation}
which again is a fraction of two expressions, both \emph{quadratic homogeneous} in $Q$. Therefore, the ratio will be the same for $cQ$ with any $c\ne 0$ and the original question can thus be rewritten as looking for the supremum of these quantities among functions in
\begin{equation*}
\mathcal{Q}^{\star} := \left\{ Q ~:~ \int_0^1 y^{k-2} Q(1-y) dy \ne 0, \quad Q(0)=0 \right\}.
\end{equation*}
Continuity of $Q'$ is not indispensable, but of course the
ratio must be a ratio of finite quantities, with a nonzero and
finite denominator, hence we need still to restrict
considerations to functions $Q' \not \equiv 0$ or, in general
allowing discontinuous functions, $Q'$ not zero almost
everywhere and also satisfying $\int_0^1 x^n Q'^2(1-x) dx <
\infty$.

We will see in a moment--see the proof of the forthcoming
Proposition \ref{p:bound}--that this latter condition also
implies that even $\int n x^{n-1} Q^2(1-x) dx <\infty$, as
needed. Furthermore, together with the restriction that
$Q(0)=0$, we see that $Q$ is constant if only $Q\equiv 0$, so
we need to exclude only this obviously singular case. Otherwise
also the numerator remans finite, i.e. the ratio \eqref{RatioQ}
exists finitely, whence $S(k)$ exists at least as a supremum of
certain finite, positive quantities.

Let us observe that the condition that $P(1)\ne 0$, is a
\emph{linear condition}, equivalently stated in the form that
the linear functional\footnote{Linearity is clear, once the
integral is defined finitely. Then again, finiteness of
$\int_0^1 y^{n-1} Q^2(1-y) dy$, appearing in the numerator of
the extremal quantity, ensures by means of the Cauchy-Schwarz
inequality, finiteness of the functional values, too. So for
the rest of the argument to be valid, it suffices to check
finiteness of the numerator of \eqref{RatioQ}.}
$Q\longrightarrow \int_0^1 Q(1-y)y^{n-1} dy$ on the function
space of admissible functions $Q$ should not vanish. In other
words, the subset which falls out of consideration for not
meeting this condition is the kernel subspace of the linear
functional, which is of codimension one--in view of the fact
that the functional itself is not identically zero, obvious
from looking at functions $Q$ with $Q|_{(0,1)}>0$ certainly
yielding positive functional values--, so a hyperplane
${\mathcal H}$ of our linear function space $\X$ (whatever
choice of the function space and respective norm we make later
on).

Therefore, dropping the restriction that $P(1)=Q(0)\ne 0$ means
only that $q \not\in {\mathcal H}$ is dropped. In the following
we will find the supremum on ${\mathcal X}$, and actually will
show that here the supremum is finite, attained at certain
maximizers.

The only issue, which may bother us a little, if the actual
maximizers will belong to the singular hyperplane ${\mathcal
H}$, or stay in $\X \setminus {\mathcal H}$. That we should
check at the end. But maximizers $q\in \X$ will actually be
positive functions, so the value of the functional $I(q)=P(1)$
will be necessarily positive for these, and maximum over $\X$
or $\X\setminus {\mathcal H}$ will thus be seen to be the same.
We will leave it to the reader to check this and from now on
pass on to the class $\X$.

\subsection{The choice of the function class of the extremal
problem \eqref{Sproblem}}\label{sec:First} In view of the
above, let us fix the function class, where the extremal
problem \eqref{Sproblem}, initiated by Soundararajan, will be
considered, as follows. Write $q(x)=Q(1-x)$ as before. Then the
whole problem becomes
\begin{align}\label{qproblem}
\text{max} ~ (k-1)\left(\int_0^1 x^{k-2} q^2(x) dx \right) \Bigg/
\left(\int_0^1 x^{k-1} q'^2(x) dx \right) \quad
\text{under condition}~~ q(1)=0,
\end{align}
understood in an appropriate function class $\X$, like, e.g.,
$C^1[0,1]$.

Partial integration in the numerator and $q(1)=0$ yields now
the reformulation
\begin{equation}\label{qMaxProblem}
S(k)=\sup_{q\in \X, q(1)=0,q\ne {\bf 0}} \left(-2 \int_0^1
x^{k-1} q'(x)q(x)  dx \right) \Bigg/ \left(\int_0^1 x^{k-1}
q'^2(x) dx \right)~.
\end{equation}
Certainly we want the denominator to be finite, so we assume
that our function class is chosen in such a way that for any
$q\in\X$ this weighted square integral of $q'$ converges. This
implies the convergence of the numerator (as we'll see soon)
and consequently that also the positive, nondegenerate linear
functional $I(q):=(k-2)! P(1) = \int_0^1 q(t) t^{k-2} dt$ is
well-defined, finite. So now we fix the largest function space
we may deal with as
\begin{equation}\label{Xdef}
\X := \left\{ q:(0,1]\to \RR ~:~ q(x)=-\int_x^1 q'(t) dt,~~
\int_0^1 x^{k-1} q'^2(x) dx < \infty \right\}.
\end{equation}
The definition above is understood to mean that any $q\in \X$
is an absolutely continuous function on each compact
subinterval of $(0,1]$, whence $q'\in L^1_{\text{loc}}(0,1]$
and $q(x)$ exists as a Lebesgue integral of $q'$, and in view
of the second condition, $q'$ is also square-integrable on
$[0,1]$ with respect to the weight $x^{k-1}$.

\subsection{An estimation of the extremal value}\label{sec:estimationS(k)}
Before proceeding let us stop for a little further analysis,
establishing bondedness of $S(k)$, because this will be needed
in what follows.

Soundararajan \cite{S} remarks that ``the unfortunate
inequality'' $S(k)<4/k$ holds. This is
not completely obvious, but in fact the situation is even
worse, namely, $S(k)<4/(k+c\log k)$. This was essentially
proved (without an explicit calculation of $c$) in the
mentioned unpublished note of Soundararajan \cite{Se}. Together
with the mentioned example $P(x)=x^{k+\ell}$, $\ell=c\sqrt{k}$,
this shows that the value of $S(k)$ is between $4/(k+c'\log k)$
and $4/(k+c''\sqrt{k}$).

\begin{proposition}[{\bf Soundararajan}]\label{p:bound} The extremal problem
\eqref{Sproblem} is bounded by $4/k$. Moreover, we have
$S(k)<\frac{4}{k+\log_2 k -5}$ for all $k\geq 4$.
\end{proposition}
\begin{proof} Let us fix, as in \eqref{RatioQ}, the value $n:=k-1$.
We are to show that whenever the denominator of \eqref{RatioQ}
exists finitely, but is nonzero (i.e. when $P^{(k)}=Q'\ne {\bf
0}$), then also the numerator (with the condition that
$Q(0)=0$, i.e. $P^{(k-1)}(0)=0$) exists finitely and, moreover,
the ratio admits the stated bounds.

Let us write now $q(x):=Q(1-x)$, assume that $Q$, hence also
$q$, are absolutely continuous, and consider the resulting
relations $q'(x)=-Q'(1-x)$, $q(1)=Q(0)=0$. These imply by
absolute continuity that $q(x)=q(x)-q(1)=-\int_x^1
q'(t)dt=\int_x^1 Q'(1-y)dy = [-Q(1-y)]_x^1 =Q(1-x)$--so we can
as well start with the conditions that $p(x):=q'(x)$ is
measurable and finite a.e., admits the weighted bound
$L:=\int_0^1 x^n q'^2(x) dx <\infty$ (coming from the
requirement that the denominator is finite), and also that $q'$
does not vanish a.e. (for the denominator being positive).
First let us check that then defining $q$ from the given
$p:=q'$ as $q(x):=-\int_x^1 q'(t)dt$ works, results in an
absolutely continuous function, and with this function the
numerator stays finite, bounded in terms of $L$.

Indeed, $\int_x^1|q'(t)|dt \leq \sqrt{\int_x^1 t^{-n} dt
\int_x^1 t^n |q'(t)|^2 dt } \leq \sqrt{ \dfrac{x^{1-n}}{n-1}
\int_0^1 t^n |q'(t)|^2 dt}
=\sqrt{\dfrac{L}{n-1}}x^{\frac{1-n}{2}}$ (valid for all $n\geq
2$) gives not only that $q(x)$, as a Lebesgue integral, exists
for all $x$, but also the estimate $|q(x)|^2\leq \dfrac{L}{n-1}
x^{1-n}$ on $(0,1]$. It follows that $q(x)$ is absolutely
continuous with derivative $p=q'$ a.e. Moreover, $\int_0^1 n
x^{n-1} q^2(x) dx \leq \dfrac{n}{n-1} L$, too, hence the
numerator in \eqref{RatioQ} is also finite and the quotient
cannot exceed $\dfrac{k-1}{k-2} \leq 2$ ($k\geq 3$).

For $k=2$, i.e. $n=1$, there is only a little difference in the
calculation, as then we obtain $|q(x)|\leq \sqrt{L |\log x|}$
and $\int_0^1 n x^{n-1} q^2(x) dx = \int_0^1 q^2(x) dx \leq L
\int_0^1 (-\log x) dx = L <2L$, extending the above bound to
all $k\geq 2$, too.

In the following we compute the stated sharper bound, too. For
any $m \geq n-1$, $I(m):= \int_0^1 x^m q^2(x) dx \leq \int_0^1
x^{n-1} q^2(x) dx \leq 2L<\infty$, as by condition we consider
the class of functions satisfying $0<L<\infty$ (with
$L:=\int_0^1 x^n q'^2(x)dx$).

Partial integration (using also $q(1)=Q(0)=0$) and
Cauchy--Schwarz inequality yield
\begin{equation*}
I(m)=\dfrac{-1}{m+1}  \int_0^1 x^{m+1} 2 q'(x)q(x) dx
\leq \dfrac{2}{m+1}\sqrt{I(2m+2-n) L}.
\end{equation*}
So starting from $m:=m_0:=k-2$ and continuing by induction with
$m_j:=k+2^j-3$ ($j=0,1,\dots,N$), we arrive at
\begin{equation*}
\frac{\int_0^1 x^{k-2} q^2(x) dx}{\int_0^1 x^{k-1} q'^2(x)dx} = \frac{I(m_0)}{L}\leq  \prod_{j=0}^{N} \left( \frac{2}{k+2^{j}-2}
\right)^{2^{-j}} \cdot \frac{I(k+2^{N+1}-3) ^{2^{-(N+1)}}} {L^{2^{-(N+1)}}}.
\end{equation*}
Since $0\leq I(\nu)$ is decreasing with $\nu$, $I(k+2^{N+1}-3)$
converges with $N$ so that we can pass to the limit
$N\to\infty$, and then even take supremum with respect to $q$,
obtaining
\begin{equation*}
S(k)\leq (k-1) \prod_{j=0}^{\infty} \left( \frac{2}{k+2^{j}-2}
\right)^{2^{-j}} = \frac{4}{\prod_{j=1}^{\infty}
\left(k+2^{j}-2\right)^{2^{-j}}} = \frac{4}{k-2}
\prod_{j=1}^{\infty}
\left(1+\frac{2^{j}}{k-2}\right)^{-2^{-j}}.
\end{equation*}
Observe that  for all $j\geq 1$ every single $k+2^{j}-2\geq k $
in the denominator of the last but one expression, hence
$S(k)\leq 4/k$ follows immediately. We can even sharpen this
estimate further. Let us denote the last product by $D:=D(k)$
and define $\ell:=[\log_2(k-2)]$. Then, by using $\log (1+x) >
x - \frac12 x^2$ (for $x>0$) we infer
\begin{align*}
\log D(k)  = - \sum_{j=1}^\infty
\frac{\log\left(1+\frac{2^{j}}{k-2} \right)}{2^j} & < -
\sum_{j=1}^{\ell} \frac{1}{2^j} \left(\frac{2^{j}}{k-2} -
\frac12\left( \frac{2^{j}}{k-2}\right)^2 \right) \\ & = -
\frac{\ell}{k-2} + \frac{2^{\ell}-1}{(k-2)^2} < -
\frac{\ell-1}{k-2}.
\end{align*}
Therefore, as $e^{-x} <\frac{1}{1+x}$ (for $x>0$), we obtain
\begin{equation*}
S(k)< \frac{4}{k-2} \cdot {\exp \left(- \frac{\ell-1}{k-2}
\right)} < \frac{4}{(k-2)\left( 1 + \frac{\ell-1}{k-2} \right)}
=\frac{4}{k +\ell-3} \leq \frac{4}{k +\log_2 k - 5},
\end{equation*}
since $\ell \geq \log_2 (k-2)
-1 \geq \log_2 k - 2$ for all $k\geq 4$.
\end{proof}

A further elementary observation is that for the
Cauchy--Schwarz inequality to be precise, we should have
$x^{m+1-n/2} q = c x^{n/2} q'$ in all the above applications of
the Cauchy--Schwarz estimate (i.e., for all occurring values of
$m$). This cannot hold for whatever choice of $q$ for all $m$
simultaneously. To have an about optimal estimate we may strive
for having the Cauchy--Schwarz estimate sharp at the very first
application, when $m=k-2$ and $n=k-1$, so $q=cq'$ follows, and
then $q(x)=e^{cx}$. But even that is not a valid choice in our
problem: $q(1)=0$ prevents us taking $q$ an exponential
function as it can never be zero. In any case, the estimate of
$S(k)$ above cannot be sharp.

\section{Existence and smoothness of maximizers in the extremal problem}\label{sec:maximum}

\subsection{Existence of maximizing functions in the extremal problem}\label{sec:eximaxi}
In this paper the role of $k$ is fixed. Furthermore, it will be
convenient for us to avoid repetitious use of $k-2$ and $k-1$,
so throughout the rest of the paper except for the last
section, Section \ref{sec:poly}, we will fix the notations for
two further integer parameters. So we define
\begin{equation}\label{mndef} m:=k-2, \qquad n:=k-1.
\end{equation}

As it is explained above, we can discuss the optimization
problem in the function space
\begin{equation}\label{Ydef}
\Y := \left\{p(t) \in L^1_{\text{loc}}(0,1] ~:~
\int_0^1 p^2(t) t^{n} dt < \infty \right\}, \quad
\textrm{where}\quad q(x)=-\int_x^1 p(t) dt .
\end{equation}
Multiplying the occurring functions by $t^{n/2}$, we can even
consider the space of functions
$\ff(t):=p(t)t^{n/2}=q'(t)t^{n/2}$, which then will be
square-integrable on $[0,1]$, so that $\ff \in
L^2[0,1]$.

Next let us establish, how the functional to be maximized looks
like over these spaces. On $\X$, on $\Y$ and finally on
$L^2[0,1]$ we must consider the respective equivalent
expressions
\begin{align}\label{maxexpressionslast}
&\frac{-2 \int_0^1 x^{k-1} q'(x)q(x)  dx }{\int_0^1 x^{k-1} q'^2(x) dx }
= \frac{2 \int_0^1 x^{n} p(x) \left( \int_x^1 p(t) dt\right) dx }{\int_0^1 p^2(x) x^{n} dx } \notag
\\ & \qquad = \frac{2 \int_0^1 \ff(x) x^{n/2} \left(\int_x^1 \ff(t) t^{-n/2} dt \right) dx}{\int_0^1 \ff^2(x) dx}
=\frac{ 2 \int_0^1 \int_0^1 \ff(x) \ff(t)  \chi_{t>x}(x,t) x^{n/2} t^{-n/2} dt dx}{\int_0^1 \ff^2(x) dx}
\notag\\ & \qquad = \frac{\int_0^1 \int_0^1 \ff(x) \ff(t)  K(x,t) dt dx}{\int_0^1 \ff^2(x) dx}
\qquad \text{with} \quad  K(x,t) := \left(\frac{\min(x,t)}{\max(x,t)}\right)^{n/2},
\end{align}
the last step being a technical one to bring the kernel $K$ to
a symmetric form.
So finally we find
that
\begin{equation}\label{SinL2}
S(k)=\sup_{L^2[0,1]\setminus \{ {\bf 0} \}}
\frac{\int_0^1 \int_0^1 \ff(x) \ff(t) K(x,t) dt dx}
{\int_0^1 \ff^2(x) dx} \quad \text{with} ~~
K(x,t) := \left(\frac{\min(x,t)}{\max(x,t)}\right)^{n/2}.
\end{equation}
Clearly, on $L^2[0,1]$ the functional in \eqref{maxexpressionslast} is defined everywhere
except $\ff = {\bf 0}$ (the zero function), and is bounded by
$4/k$, as proved before. Moreover, there is a clear homogeneity
property: The ratio for any $\ff$ is equal to the ratio for any
nonzero constant multiple $c\ff$, hence the ratio is constant
on all rays $\{ c\ff~:~ c\in\RR, c\ne 0\}$.

Therefore, the range of this quotient functional is clearly the
same on the whole space $L^2[0,1]\setminus \{ {\bf 0} \}$ as on
$B \setminus \{\bf 0\}$, $B$ denoting the unit ball $B:=\{
\ff\in L^2[0,1]~:~ \|\ff\|_2 \leq 1 \}$ (where the $2$-norm of
a function in $L^2[0,1]$ is $\|\ff\|_2:=({\int_0^1 \ff^2
(x) dx})^{1/2}$, as usual). Furthermore, actually already
\emph{on the unit sphere} $S:=\{ \ff\in L^2[0,1]~:~ \|\ff\|_2 =
1 \}$ the functional must take on all the values of its range.
However, on the unit sphere the denominator is exactly one, so
now we can modify the formulation and write
\begin{equation*}
S(k) = \sup_S  \int_0^1 \int_0^1 \ff(x) \ff(t)  K(x,t) dt dx =
\sup_{B} \int_0^1 \int_0^1 \ff(x) \ff(t)  K(x,t) dt dx.
\end{equation*}
Moreover, it is clear that in this
last formulation $S(k)$ is taken by a maximizer function $\ff
\in L^2[0,1]$ iff there is a maximum at some $\ff \in S$ iff
there is a maximum on $B$ (in which case again any maximum must
belong to $S$). So any maximizer $\ff$ in the original
formulation is maximizer together with all the ray $\{c\ff\}$
of its homothetic copies, and in the new formulation this
maximizer occurs exactly with $c=\pm 1/\|\ff\|_2$, i.e., at the
unit norm elements of the given ray.

This reformulation furnishes us the access to settle the
existence question of some maximizer. In our formulation of the
extremal problem  all functions are real-valued, only for the
next two propositions (spectral theory), and for the sake of
being precise, we shall need complex-valued functions.

\begin{proposition}
Let
\begin{equation*}
K(x,y) := \left(\frac{\min(x,y)}{\max(x,y)}\right)^{n/2},
\end{equation*}
and define the Fredholm-type operator
\begin{equation}\label{Tdef}
T : L^2([0,1];\CC)\to L^2([0,1];\CC)\quad (T\ff)(x):= \int_0^1 \ff(y) K(x,y) dy.
\end{equation}
Then $T$ is a compact, positive, self-adjoint operator on the complex Hilbert space $L^2([0,1];\CC)$, maps real-valued functions into real-valued ones, and preserves positivity.
\end{proposition}
\begin{proof}
Since $K\in L^\infty([0,1]\times [0,1])$, $T$ is compact, see \cite[\S97]{RSz}. Since $0\leq K\leq 1$  and $K$ is symmetric, the other two properties follow evidently.
\end{proof}

\begin{proposition}\label{p:maximum} $S(k)$ is attained as a maximum by some maximizing function $\ff\in L^2[0,1]$.

\medskip\noindent Equivalently,
\begin{equation*}S(k)=\frac{-2 \smallint_0^1 x^{k-1} q'(x)
q(x)  dx }{\smallint_0^1 x^{k-1} q'^2(x) dx}\end{equation*} for
some appropriate $q\in \X$ with $q(1)=0$ and $q\not\equiv {\bf
0}$.
\end{proposition}
\begin{proof}
Consider the operator $T$ as in \eqref{Tdef}.
\begin{equation}\label{Adef}
A(\ff, \psi) := \int_0^1 \int_0^1 \ff(y) \overline{\psi}(x) K(x,y) dxdy=\langle T\ff,\psi\rangle,
\end{equation}
which is a sesquilinear form on $L^2([0,1];\CC)$. By \cite[\S93]{RSz} we have
\begin{equation*}
\|T\| := \sup_{\|\ff\|_2\leq 1} \|T\ff\| = \sup_{\|\ff\|_2\leq 1} |\langle T\ff,\ff \rangle| = \sup_{\|\ff\|_2\leq 1} A(\ff,\ff).
\end{equation*}
Since $T$ is compact, positive and self-adjoint, all of its
eigenvalues are nonnegative, moreover, the eigenvalues can be
ordered in a decreasing null-sequence $(\lambda_j)$,
$\lambda_1>\dots>\lambda_j>\dots$, $\lambda_j\to 0$
($j\to\infty$), and we also have
\begin{equation*}
\|T\|=\max\{\EV~:~\mbox{$\EV$ is an eigenvalue of $T$}\}=:\EV_{1}.
\end{equation*}
Since $T$ leaves the subspace of real-valued functions
invariant, for any eigenvalue $\EV\in \RR$ of $T$ there is a
real-valued eigenfunction. Summing
up, 
$\|T\|=\EV_1$, and there exists some (nonzero) eigenfunction
$\ff\in L^2[0,1]$ satisfying $\|\ff\|=1$ and $\EV_1= \|T\| =
A(\ff,\ff)$, yielding a maximizer for $A(\ff,\ff)$ as asserted.
\end{proof}
\begin{remark}\label{uniqremark}
The above proof yields also the following important
information: $S(k)$ is { the largest} eigenvalue $\lambda_1$ of
$T$, and any (normalized) eigenfunction $\ff$ of $T$ {belonging to $\lambda_1$} is a maximizer; moreover, the only
maximizers are nonzero eigenfunctions of $T$ corresponding to
$\EV_1=\|T\|$.

{ Indeed, as} $T$ is compact and self-adjoint, there is an
orthonormal basis $(e_j)$ in $L^2[0,1]$ that consists of
eigenfunctions of $T$. Let $\mathbf{0}\neq\varphi\in L^2[0,1]$
be not an eigenfunction to the eigenvalue $\EV_1$. Then
$\ff=\sum_{j=1}^\infty \langle \ff, e_j\rangle e_j$ and
\begin{equation*}
\langle T\ff,\ff \rangle=\sum_{j=1}^\infty\EV_j
|\langle \ff, e_j\rangle|^2<\EV_{1}\|\ff\|_2
\end{equation*}
by Parseval's identity, where for the strict inequality ``$<$''
we have used that for some $j>1$ we have $|\langle \ff,
e_j\rangle|>0$, while $\EV_j<\EV_1$. 

We also remark that since $K$ is
strictly positive, so is the operator $T$, hence one knows from
Perron--Frobenius theory (see \cite[Sec.~4.2]{M}) that the dominant eigenvalue $\lambda_1$ is simple with a corresponding strictly positive  eigenfunction. This will be proved
later also by directly determining all eigenfunctions of $T$.
\end{remark}
Next we turn to smoothness properties of eigenfunctions of $T$.
\subsection{Smoothness of maximizers and maximizers in $C[0,1]$}\label{sec:smooth}
The above formulation also provides us a direct access to further smoothness statements.

\begin{lemma}\label{l:range} The Fredholm-type operator $T$ defined in \eqref{Tdef} maps
$L^2[0,1]$ to the subspace $C_0(0,1]$ of continuous functions
with value $0$ at $0$\footnote{Note that we identify functions
defined on $(0,1]$ only but having limit 0 towards the boundary
point 0 with functions continuously extended to 0 by defining
their value at 0 as 0.}.
\end{lemma}
\begin{proof}
Since $L^2[0,1]\subset L^1[0,1]$, and $0\leq K(x,y)\leq 1$, the
expression $(T\ff)(x)= \int_0^1 \ff(y) K(x,y) dy$ is an
integral with a uniform majorant $|\ff(y)|\in L^1[0,1]$ of the
integrands. Hence by the Lebesgue Dominated Convergence
Theorem, it suffices to take the pointwise limit under the
integral sign. When $x\to x_0$, this gives for all $y>0$
$\lim_{x\to x_0} K(x,y)=K(x_0,y)$, while for $y=0$ we have
$K(x,0)=K(x_0,0)=0$ identically. (Essentially, we have used
only separate continuity of $K$ on $[0,1]\times [0,1]$.) Thus
$\lim_{x\to x_0} (T\ff)(x)=\int_0^1 \ff(y) \lim_{x\to x_0}
K(x,y) dy=\int_0^1 \ff(y) K(x_0,y) dy=(T\ff)(x_0)$, i.e.,
$T\ff\in C[0,1]$. By definition, $K(0,y)=0$ for all
$y\in[0,1]$, hence for every $\ff\in L^2[0,1]$ we have
$(T\ff)(0)=0$.
\end{proof}

\begin{corollary}\label{p:contmax}
All eigenfunctions $\ff$ of the Fredholm-type operator $T$ defined in \eqref{Tdef} are continuous and fulfill
$\ff(0)=0$.

\medskip\noindent
Equivalently, in the function space $\X$ defined in
\eqref{Xdef} all the functions $q(x)=-\int_x^1
q'(t)dt=-\int_x^1 t^{-n/2} \ff (t) dt$ corresponding to
eigenfunctions $\ff$ of $T$ satisfy
$\ff(x)=x^{n/2} q'(x) \in C_0(0,1]$ and thus $x^{n/2-1}q(x) \in
C_0(0,1]$.
\end{corollary}

\begin{proof} All eigenfunctions lie in the range of the operator $T$,
hence belong to $C_0(0,1]$ in view of Lemma \ref{l:range}.

\medskip\noindent
Recall that the correspondence between $L^2[0,1]$ and our
spaces $\Y$ and $\X$ was given by $\ff(x)= x^{n/2} p(x) =
x^{n/2} q'(x)$. Thus for $\ff\in L^2[0,1]$, an eigenfunction of
$T$, we obtain for the corresponding $q$ that $x^{n/2}q'(x)\in
C_0(0,1]$, whence also $q'\in C(0,1]$ follows. Moreover,
$\lim_{x\to 0+} x^{n/2} q'(x)= \lim_{x\to 0+} \ff(x) =
\ff(0)=0$, providing a continuous extension of $x^{n/2}q'(x)$
even to $0$. Now writing $q(x)=-\int_x^1 q'(t) dt$ yields $q\in
C(0,1]$. While for $x\to 0+$ we obtain that\footnote{Here
$o(1)$ means that for any $\ve>0$ we have some $\de$ such that
for $0\leq t\leq \delta$, $|q'(t)|<\ve t^{-n/2}$. Therefore, we
have for any $0<x<\de$ the estimate $|q(x)-q(\de)|\leq
\int_x^\de \ve t^{-n/2} dt < \dfrac{2\ve}{n-2} x^{1-n/2}$,
while $x^{n/2-1}q(\de) \to 0$ as $x\to 0+$. Therefore,
$|x^{n/2-1}q(x)|\leq \dfrac{2\ve}{n-2} +o(1)$ and finally
$x^{n/2-1}q(x)\to 0$ with $x\to 0+$.}
\begin{equation*}
\lim_{x\to 0+} x^{n/2-1} q(x) = \lim_{x\to 0+} x^{n/2-1}
\left( -\int_x^1 o(1)t^{-n/2} dt\right) = \lim_{x\to 0+} o(1) \frac{2}{n-2} = 0,
\end{equation*}
hence $x^{n/2-1} q(x) \in C_0(0,1]$.
\end{proof}
Although $K$ is not everywhere continuous on
$[0,1]\times[0,1]$, the operator $T$ can still be approximated
by compact operators given by continuous kernels.
\begin{proposition} The operator $T$ restricted to $C[0,1]$ is a compact $C[0,1]\to
C[0,1]$ operator with exactly the same eigenvalues and eigenfunctions as on $L^2[0,1]$.
\end{proposition}
\begin{proof}
Let $f_j:[0,1]\times [0,1]\to [0,1]$ be continuous with
$f_j(x,y)=0$ if $x,y\leq \frac{1}{2j}$ and $f_j(x,y)=1$
if $\max(x,y) \geq \frac1{j}$. Then $f_j K$ is continuous,
hence the integral operator $T_j$ with kernel $f_j K$ is
compact, see, e.g., \cite[\S90]{RSz}. It is easy to see that
$T_j\to T$ in the operator norm (over $C[0,1]$), hence $T$
itself is compact, see \cite[\S76]{RSz}.

\medskip\noindent By Proposition \ref{p:contmax} all eigenfunctions of $T$ on
$L^2[0,1]$ belong also to $C_0(0,1]$, hence remain
eigenfunctions when $T$ is considered as $C[0,1]\to C[0,1]$.
The converse is obvious: Every continuous eigenfunction is of
course also an eigenfunction from $L^2[0,1]$. In particular,
the set of eigenvalues are also exactly the same when
considered in these two spaces.
\end{proof}
\begin{remark}
One can show that the norm of $T$ as an operator on $C[0,1]$ is
\begin{equation}\label{eq:Tnorm}
\|T\|_{C[0,1]}=\left(\frac{4}{n+2}\right)^{\frac{n}{n-2}}=\left(\frac{4}{k+1}\right)^{\frac{k-1}{k-3}}.
\end{equation}
In fact, since $T$ is positivity preserving, we have
$\|T\|_{C[0,1]}=\|T \mathbf{1}\|_\infty$, where $\mathbf{1}$ is
the constant $1$ function. Easy calculation shows that
$(T\mathbf{1})(x)=\frac{4n}{n^2-4}x-\frac{2}{n-2}x^{n/2}$ and
that this function has maximum at
$x=(\frac{2}{n+2})^{2/(n-2)}$. Since we already know
$\|T\|_{L^2[0,1]}=\EV_{1}$, and in general $\EV_{1}\leq
\|T\|_{C[0,1]}$, we obtain that the maximum of $A(\ff,\ff)$
with $\|\ff\|_2=1$, i.e., $S(k)$ is smaller than the constant
in \eqref{eq:Tnorm} above.
\end{remark}

\subsection{Differentiability of maximizers}\label{sec:diffmax}
We now push further the smoothness statements from the last
subsection. We need some preparations, and define the following
auxiliary functions
\begin{equation}\label{kappadef}
\kappa(x,y):=\begin{cases} \frac{\min(x,y)}{\max(x,y)} \qquad
&\text{if} \quad (0,0)\ne (x,y) \in [0,1]^2 ,\\
 0 & \text{if} \quad (x,y)=(0,0), \end{cases}
\end{equation}
and
\begin{equation}\label{omdef}
\om(a,b):= \begin{cases}\frac{\sum_{j=0}^{n-1}
a^{j/2}b^{(n-1-j)/2}}{\sqrt{b}+\sqrt{a}} \qquad & \text{if}
\quad (0,0)\ne (a,b)\in [0,1]^2, \\ 0 & \text{if} \quad
(a,b)=(0,0). \end{cases}
\end{equation}
With these notations we have for every $0<x,y\leq 1$ the formula
\begin{equation}\label{Kkappaomega}
\frac{n}{2}K(x,y)=\omega(\kappa(x,y),\kappa(x,y)) \kappa(x,y),
\end{equation}
which also holds for $x$ or $y$ being $0$, with both sides
vanishing.

Now note that $0 \leq \om(a,b) \leq n \max(a,b)^{n/2-1} \leq
n$. Furthermore, observe that for $0\leq a, b\leq 1$ ({even if}
both are zero) we have $b^{n/2}-a^{n/2} =
(\sqrt{b}-\sqrt{a})\sum_{j=0}^{n-1} a^{j/2}b^{(n-1-j)/2} =
(b-a)\om(a,b)$. Hence { for any $y,x,x'\in[0,1]$} we can write
\begin{align}\label{Kkappaom}
K(x',y)-K(x,y)
&=\om\big(\kappa(x',y),\kappa(x,y)\big) \cdot \big(\kappa(x',y) - \kappa(x,y) \big).
\end{align}
Fix $x>0$. Denoting $\D:=x'-x>0$, we also have
\begin{equation}\label{kappa}
\kappa(x',y)-\kappa(x,y) =\begin{cases}
\frac{y}{x'}-\frac{y}{x} = - \frac{y\D}{xx'} \qquad\qquad &
\text{if}~~ y<x<x', \\ \frac{y}{x'}-\frac{x}{y}=
\frac{y^2-xx'}{yx'} \qquad & \text{if}~~ x \leq y\leq x', \\
\frac{x'}{y}-\frac{x}{y} = \frac{\D}{y} \qquad\qquad\qquad &
\text{if}~~ x<x'<y,
\end{cases}
\end{equation}
whence
\begin{equation}\label{kappaesti}
|\kappa(x',y)-\kappa(x,y)| =\begin{cases}
\frac{y}{x}-\frac{y}{x'} = \frac{y\D}{xx'}\leq \frac{\D}{x'}\leq \frac{\D}{x}    \qquad \qquad\qquad
& \text{if}~~ y<x<x', \\ \left|\frac{x}{y}-\frac{y}{x'}
\right|= \frac{|xx'-y^2|}{yx'} \leq \frac{\D}{y}\leq \frac{\D}{x}   \qquad &
\text{if}~~ x \leq y\leq x', \\ \frac{x'}{y}-\frac{x}{y} =
\frac{\D}{y} \leq \frac{\D}{x'} \leq \frac{\D}{x}  \qquad\qquad\qquad & \text{if}~~ x<x'<y.
\end{cases}
\end{equation}
\begin{lemma}\label{l:C1range} The operator $T$
maps $L^2[0,1]$ to the space $C^1(0,1]$. Furthermore, for $\ff\in L^2[0,1]$ and $x\in(0,1]$
\begin{equation}\label{Tffprime}
(T\ff)'(x) = \int_0^1 \ff (y) \frac{\partial}{\partial x} K(x,y) dy
= - \frac{n}{2x} (T\ff)(x) + n x^{n/2-1}\int_x^1 \frac{\ff(y)}{y^{n/2}} dy.
\end{equation}
\end{lemma}
\begin{proof} The two expressions given for $(T\ff)'(x)$ in \eqref{Tffprime}
are easily seen to be equal, so the proof hinges upon showing that $(T\ff)'(x)$
equals any one of them.

\medskip\noindent For any $0\leq x < x'\leq 1$ using \eqref{Kkappaom} we can write
\begin{align}\label{Tfidifferential}
\frac{(T\ff)(x')-(T\ff)(x)}{x'-x} & =
\int_0^1 \ff(y) ~ \om\big(\kappa(x',y),\kappa(x,y)\big)
\frac{\kappa(x',y) - \kappa(x,y)}{x'-x} ~ dy.
\end{align}
We fix $x_0>0$ and take  either $x=x_0$ and $x'\to x_0+$, or
$x'=x_0$ and $x \to x_0-$. In any case, by \eqref{kappaesti}
and $0\le \om \le n$ we have the Lebesgue integrable majorant
$n |\ff(y)|/x_0$ of the integrand, thus limit and integral can
be interchanged.

For example in the case $x'\to x_0+$ taking into account
\eqref{Kkappaomega} we are led to
\begin{align}\label{Tfidifflimit}
\lim_{x'\to x_0} & \frac{(T\ff)(x')-(T\ff)(x_0)}{x'-x_0}  = \int_0^1 \ff(y) \om\big(\kappa(x_0,y),\kappa(x_0,y)\big) \frac{\partial}{\partial x} \kappa(x,y)\big|_{x=x_0} dy\notag \\
 & = \int_0^1 \ff(y) \om\big(\kappa(x_0,y),\kappa(x_0,y)\big) \sign (y-x_0) \frac{\kappa(x_0,y)}{x_0} dy\notag \\
& = \frac{n}{2x_0} \int_0^1 \ff(y) K(x_0,y) \sign (y-x_0) dy 
= \frac{n}{2x_0} \left( \int_{x_0}^1 \ff(y) K(x_0,y) dy - \int_0^{x_0} \ff(y) K(x_0,y) dy \right)\notag \\
& = \frac{n}{x_0} \left(\int_{x_0}^1 \ff(y) K(x_0,y) dy -\frac12 (T\ff)(x_0) \right),
\end{align}
where we have used  $\frac{\partial}{\partial x} \kappa(x,y) =
\frac{\partial}{\partial x} (x/y) = 1/y$ for $y>x$ and
$\frac{\partial}{\partial x} \kappa(x,y) =
\frac{\partial}{\partial x} (y/x) = -y/x^2$ for $x<y$. When
substituting the definition of $K$ in the above, we obtain all
the asserted formulas. Note that in case $y=x_0$, one sided
derivatives of $\kappa(x_0,y)$ still exist (and are equal to
the limits from the respective side) but the existence and
value of the limit at one exceptional point does not interfere
the value of the integral, therefore we have just put $0$ for
the value of $\frac{\partial}{\partial x}
\kappa(x,x_0)\big|_{x=x_0}$ here.

When $x\to x_0-=x'-$, the  calculation is entirely the same.

The integrals on the right hand side of \eqref{Tffprime} are of
course integrals of integrable functions, and as such, are
continuous in function of the limits of integration. Therefore,
continuity of $(T\ff)'$ on $(0,1]$ also follows.
\end{proof}

\begin{remark}\label{rem:rightderivative}
When $x_0=0$, only the
right hand side derivative can be considered and thus we take
$x_0=x=0$ and $x'\to 0+$. Also, $(T\ff)(0)=0$ and $K(0,y)=0$,
hence the consideration of the differential reduces to
\begin{equation}\label{Tprime0}
\lim_{x'\to 0+} \frac{(T\ff)(x')}{x'} =
\lim_{x'\to 0+} \frac{1}{x'}\int_0^1 \ff(y) K(x',y) dy =
\lim_{x'\to 0+} \int_0^1 \ff(y)  \frac{\kappa^{n/2}(x',y)}{x'} ~ dy,
\end{equation}
which, however, cannot be handled for general $\ff\in L^2[0,1]$
or not even for $\ff\in C_0(0,1]$, and can be well estimated
only if we use something more on $\ff$. See { Corollary
\ref{c:Tffiprimeat0} below.}
\end{remark}
\begin{proposition}\label{p:Tfftotvar}
The operator $T$ maps $L^2[0,1]$ to the space of absolutely
continuous functions with bounded total variation. Moreover,
for the total variation of $T\ff$ we have $V(T\ff,[0,1])
\leq 2\|\ff\|_1$.
\end{proposition}
\begin{proof}
We already know that $T\ff \in C^1(0,1]$, so the total
variation, whether finite or infinite, can be computed as
$V(T\ff,[0,1])=\int_0^1 |(T\ff)'|$. Now the first formula
from \eqref{Tffprime} furnishes
\begin{align*}
V(T\ff,[0,1])&=\int_0^1 |(T\ff)'(x)|dx = \int_0^1 \left|
\int_0^1 \ff(y) \frac{\partial}{\partial x} K(x,y) dy \right| dx
\\ & \leq \int_0^1 \int_0^1 |\ff(y) | \left|\frac{\partial}{\partial x}
K(x,y) \right|dydx = \int_0^1 |\ff(y) | \left( \int_0^1
\left|\frac{\partial}{\partial x} K(x,y) \right|dx \right) dy
\leq 2 \|\ff\|_1,
\end{align*}
as for all $0<y<1$ fixed we have $ \int_0^1
\left|\frac{\partial}{\partial x} K(x,y) \right|dx =
V(K(\cdot,y),[0,1])=2-y^{n/2}<2$.
\end{proof}

\begin{corollary}\label{c:Tffiprimeat0} If $\psi$ lies in the range
of $T$, 
then { $T\psi\in C^1[0,1]$, and $(T\psi)'(0)=0$. In particular
if $\ff$ is an eigenfunction of $T$, then $\ff\in C^1[0,1]$,
and $\ff'(0)=0$.}
\end{corollary}
\begin{proof}
We have to calculate the limit in \eqref{Tprime0} for $\psi:=T\ff$ in
place of $\ff$. Recall from the above that then $\psi=T\ff\in C_0(0,1]$, in particular $\psi(0)=(T\ff)(0)=0$, and $\psi\in C^1(0,1]$, $V(\psi,[0,1])\leq 2\|\ff\|_1$. The second mean value theorem
and integration by parts yield with some appropriate
$z:=z_{x'}\in (0,x')$
\begin{align*}
\frac{1}{x'} \int_0^1 & \psi(y) K(x',y) dy = \frac{1}{x'}
\left\{ \int_0^{x'} \psi(y) K(x',y) dy + \int_{x'}^1 \psi(y)
\frac{x'^{n/2}}{y^{n/2}}dy \right\}
\\&= \psi(z_{x'}) K(x',z_{x'}) +  \left\{ \left[\psi(y)
\frac{-x'^{n/2-1} }{n/2-1} y^{1-n/2}\right]_{x'}^1 + \int_{x'}^1  \psi'(y)
\frac{x'^{n/2-1} }{n/2-1} y^{1-n/2} dy \right\},
\end{align*}
so the first term tends to $\psi(0)=0$ when $x'\to 0+$. The term in
the square bracket contributes $x'^{n/2-1}\frac{-\psi(1)}{n/2-1} +
\frac{\psi(x')}{n/2-1} $, and as $x'\to 0+$ and $\psi(x')\to
\psi(0)=0$, both terms converge to 0. Finally, for the integral
Proposition \ref{p:Tfftotvar} gives that $\psi' \in L^1[0,1]$,
while the product of the further factors stays bounded uniformly for
all $x',y\in [0,1]$, as the integral runs only through values $y\ge
x'$. That is, we can again use the Lebesgue Dominated Convergence
Theorem and calculate the limit by moving it under the integral
sign. Furthermore, the pointwise limit of the expression is zero for
all fixed $y$, whence the assertion follows.
\end{proof}


\section{Solving the maximization problem}\label{sec:DiffEq}

\subsection{Setting up a differential equation for potential
extremal functions}\label{sec:setupdiffeq} By the previous
section we know that our maximization problem has a solution,
and we also saw that maximizers are sufficiently smooth. We can
now set up a differential equation to find maximizers, or which
is essentially equivalent, to find the eigenfunctions of $T$.

\begin{proposition}\label{p:fiq}
Let $\ff \in L^2[0,1]$ be an eigenfunction of $~T$
corresponding to the eigenvalue $\EV>0$. Then $\ff$ is
continuous on $[0,1]$, infinitely often differentiable on
$(0,1]$. The function $q(x)=-\int_x^1\ff(y)y^{-n/2}dy$
satisfies $q(1)=0$ and the differential  equation
\begin{equation}\label{qdiffeq}
q''(x) + \frac{n}{x} q'(x) + \frac{b}{x} q(x)=0\quad
\Big( x\in(0,1] \Big), \qquad \text{\rm where} \quad b:=\frac{n}{\EV}>0.
\end{equation}
Conversely, let $\EV>0$ and suppose that $q$ is a nonzero,
$C^2(0,1]$ solution of the differential equation above with
$q(1)=0$. If $\ff(x)=x^{n/2}q'(x)$ extends continuously to $0$
with $\lim_{x\to 0+} x^{n/2} q'(x)=0$, then
$\ff$ is an eigenfunction of $T$ corresponding to the eigenvalue
$\EV>0$.
\end{proposition}
\begin{proof}
If $\ff\in L^2[0,1]$ is an eigenfunction of $T$ for the eigenvalue $\EV>0$, then it belongs to the range of $T$, hence
is continuous and continuously differentiable on $(0,1]$ by Lemma \ref{l:C1range}.
Substituting $T\ff=\EV\ff$ in \eqref{Tffprime} we obtain
\begin{equation*}
\EV \ff'(x)=- \frac{n}{2x} \EV \ff(x) + nx^{n/2-1}\int_x^1 \frac{\ff(y)}{y^{n/2}} dy.
\end{equation*}
As the right-hand side is differentiable, we can differentiate
also the left-hand side showing $\ff \in C^2(0,1]$. We substitute $x^{-n/2}\ff(x)=q'(x)$ and $\ff(x)=x^{n/2}q'(x)$ and obtain
\begin{equation*}
\frac{d}{dx}\left(\EV x^{n/2}q'(x)\right)=- \frac{n}{2x}\EV x^{n/2}q'(x)+nx^{n/2-1}\int_x^1q'(y) dy,
\end{equation*}
and hence
\begin{equation*}
\frac{d}{dx}\left(\EV x^{n/2}q'(x)\right)=- {\EV }\frac{n}{2}x^{n/2-1}q'(x)-nx^{n/2-1}q(x).
\end{equation*}
Differentiation yields
\begin{equation*}
\EV x^{n/2} q''(x) + \EV \frac{n}{2} x^{n/2-1} q'(x)= - \frac{\EV n}{2x}x^{n/2}q'(x)-nx^{n/2-1}q(x),
\end{equation*}
and then by rearranging we obtain
\begin{equation*}
\EV x^{n/2} q''(x) + \EV \frac{n}{x} x^{n/2} q'(x) + n x^{n/2-1} q(x)=0.
\end{equation*}
Division by $\EV x^{n/2}$ thus leads to the asserted differential equation \eqref{qdiffeq}.

\medskip\noindent To see the converse we set $\psi=T\ff$. Note
that then $\psi\in C_0(0,1]$ according to Lemma \ref{l:range}.
Then { $\ff(x)=x^{n/2}q'(x)$ entails}
\begin{equation*}
q''(x)=-\frac n2 x^{-n/2-1}\ff(x)+x^{-n/2}\ff'(x),
\end{equation*}
so that { using the assumption that $q$ solves \eqref{qdiffeq}
we obtain}
\begin{equation*}
-\frac n2 x^{-n/2-1}\ff(x)+x^{-n/2}\ff'(x)+\frac{n}{x}
x^{-n/2}\ff(x) + \frac{n/\EV}{x} \left(- \int_x^1y^{-n/2}\ff(y)dy\right)=0,
\end{equation*}
and thus also
\begin{equation*}
\frac n2\EV  x^{-n/2-1}\ff(x)+\EV x^{-n/2}\ff'(x)-\frac{n}{x}
\int_x^1y^{-n/2}\ff(y)dy=0.
\end{equation*}
By \eqref{Tffprime} { with $\psi=T\ff$ we also have}
\begin{equation*}
x^{-n/2}\psi'(x) = - \frac{n}{2}x^{-n/2-1} \psi(x) + \frac nx
\int_x^1 \frac{\ff(y)}{y^{n/2}} dy,
\end{equation*}
so
\begin{equation*}
\frac n2\EV  x^{-n/2-1}\ff(x)+\EV x^{-n/2}\ff'(x)-x^{-n/2}\psi'(x)-\frac{n}{2}x^{-n/2-1} \psi(x)=0.
\end{equation*}
If we multiply by $x^n$, we obtain for all $x>0$
\begin{equation*}
0=\frac n2\EV  x^{n/2-1}\ff(x)+\EV
x^{n/2}\ff'(x)-x^{n/2}\psi'(x)-\frac{n}{2}x^{n/2-1}
\psi(x)=\frac{d}{dx}\left(x^{n/2}\EV
\ff(x)-x^{n/2}\psi(x)\right).
\end{equation*}
Since $\ff,\psi\in C[0,1]$, $x^{n/2}(\EV\ff(x)-\psi(x))$
must vanish at $0$, whence $\EV \ff=\psi$ follows.
\end{proof}

Thus the solution of the maximization problem is reduced to
solving the homogeneous second order ordinary differential
equation \eqref{qdiffeq}, and to finding the feasible values of
$\EV$.  Next we solve this equation and analyze some
properties of the solutions.

\subsection{Bessel functions and Bessel's differential
equation}
Recall Bessel's differential equation
\begin{equation}\label{BesselEq}
y''(x)+ \frac{1}{x} y'(x) + \left(1- \frac{\nu^2}{x^2}\right) y(x) =0,
\end{equation}
for some fixed parameter $\nu\in \RR$. The Bessel function
$J_\nu$ (of the first kind) is a solution of this equation,
\cite[(6.71), p. 115]{B}. Notice that $J_{-\nu}$ is also a
solution, but for $\nu$ integer $J_\nu$ and $J_{-\nu}$ are
linearly dependent, in fact $J_{-\nu}=(-1)^\nu J_{\nu}$. So for
$\nu\in \NN$ another, linearly independent solution is needed,
which is provided by the Bessel functions $Y_\nu$ (of the
second kind), obtained as
\begin{equation*}
Y_\nu(x):=J_\nu(x)\int \frac{dx}{xJ_\nu^2(x)},
\end{equation*}
where specifying the limits of integration  is equivalent to
fix some primitive of the integrand, and the integration limit
cannot be at $0$  where $Y_\nu(x)$ is divergent in the order
$x^{-\nu}$, see \cite[(6.73), (6.74)]{B}. Then for $\nu\in \NN$
the general solution of equation \eqref{BesselEq} is a linear
combination $c_1 J_\nu + c_2 Y_\nu$, see \cite[\S 102]{B}.

\subsection{Solution of the differential equation}\label{sec:WM}
Bowman computes, see \cite[(6.80), p. 117]{B}, what happens if we consider the transformed,
substituted functions $u(x):=x^\alpha y (\beta x^\gamma)$,
where $y$ satisfies the Bessel equation \eqref{BesselEq}, and
establishes that then the new functions $u(x)$ will be the
general solutions of the transformed equation
\begin{equation}\label{Bowman}
u''(x) - \frac{2\alpha -1}{x} u'(x) + \left(\beta^2\gamma^2
x^{2\gamma-2} + \frac{\alpha^2 - \nu^2\gamma^2}{x^2} \right) u(x)
=0.
\end{equation}
If we choose here the parameters $\nu:=m$,
$\alpha:=-m/2=1-k/2$, $\beta:=2 \sqrt{n/\EV}$ and $\gamma:=1/2$
(where $n:=k-1=m+1$ and $m:=k-2$ as fixed above in
\eqref{mndef}), then the equation \eqref{Bowman} becomes
exactly \eqref{qdiffeq}. Thus we obtain that for any fixed
values of $m:=k-2$ and $\EV> 0$, there is a one-to-one
correspondence between the solutions $q$ of \eqref{qdiffeq} and
$y$ of \eqref{BesselEq} given by $q(x):=x^\alpha y (\beta
x^\gamma)=x^{1-k/2} y (2 \sqrt{n/\EV} \sqrt{x})$.
\begin{corollary} Every solution $q$ of \eqref{qdiffeq} is
a linear combination of transformed Bessel functions from the
above, i.e.,
\begin{equation*}
q(x)=c_1x^{-m/2}J_m(2\sqrt{n/\EV}
\sqrt{x})+c_2x^{-m/2}Y_m(2\sqrt{n/\EV} \sqrt{x}),\quad
(m=n-1=k-2).
\end{equation*}
\end{corollary}

\subsection{Some analysis of the occurring Bessel type functions}\label{sec:Bessel}
Before proceeding, let us note one important thing. Not all
solutions of the differential equation are relevant for us,
because the resulting $q'$ must have finite  weighted  square
integral, i.e.,  $\ff(x)=x^{n/2}q'(x)$ has to belong to
$L^2[0,1]$ (and even be continuous on $[0,1]$, according to
Proposition \ref{p:fiq}).

According to the second formula of \cite[9.1.30]{AS} (with the choice $k=1$
and $\nu=m$ there) we have
\begin{align*}
&\left(\frac{1}{x} \frac{d}{dx} \right) \left( x^{-m}
c_1J_m (x)+c_2Y_m(x)\right)\\
&\qquad=-x^{-m-1}\left(c_1J_{m+1} (x)+c_2Y_{m+1}(x)\right)=-x^{-n}\left(c_1J_n (x)+c_2Y_n(x)\right),
\end{align*}
therefore we obtain for a solution $q$  of \eqref{qdiffeq} that
\begin{align*}
q'(x) &=\frac{d}{dx}\left(c_1x^{-m/2}J_m(2\sqrt{n/\EV}
\sqrt{x})+c_2x^{-m/2}Y_m(2\sqrt{n/\EV} \sqrt{x})\right)\\
&=-\frac{1}{2\sqrt{n/\EV}} x^{1/2}x^{-n/2}\left(c_1J_n(2\sqrt{n/\EV}
\sqrt{x})+c_2Y_n(2\sqrt{n/\EV} \sqrt{x})\right).
\end{align*}
So that
\begin{equation}\label{ffiJnYn}
\ff(x)=x^{n/2} q'(x) =-
\frac{\sqrt{x}}{2\sqrt{n/\EV}}
\left(c_1J_n(2\sqrt{n/\EV}
\sqrt{x})+c_2Y_n(2\sqrt{n/\EV} \sqrt{x})\right).
\end{equation}
To see when such a $\ff$ may actually be an eigenfunction of
$T$, we first find out when it belongs to $C_0(0,1]$. First,
see \cite[(1.2)]{B}
\begin{equation}\label{Besselseries}
J_\nu(x)=x^\nu \sum_{j=0}^{\infty} \frac{(-1)^j}{2^{\nu+2j}}
\frac{x^{2j}}{j!(\nu+j)!}, \quad \textrm{whence also}\quad
J_\nu(x)\sim \frac{1}{2^\nu \nu!} x^\nu~~(x\to 0+),
\end{equation}
so in particular $J_n$ is continuous on $[0,\infty)$, and for
any $c_1\in \RR$ the part $\frac{\sqrt{x}}{2\sqrt{n/\EV}}
c_1J_n(2\sqrt{n/\EV} \sqrt{x})$ belongs to $C_0(0,1]\subset
L^2[0,1]$. Second, $Y_n(x) \asymp x^{-n}$ ($x\to 0+$) \cite[p.
116]{B}, entailing that for $c_2\ne 0$
$\frac{\sqrt{x}}{2\sqrt{n/\EV}} c_2 Y_n(a\sqrt{x})\asymp
x^{-(n-1)/2}$, and this function is not even bounded near $0$,
if $n>1$ (i.e., when $k=n-1>2$). If $n=1$, i.e., $k=2$, this function is not vanishing at $0$, a condition that is necessary for an eigenfunction of $T$ by Lemma \ref{l:range}. So from Corollary \ref{p:contmax} we obtain that
$\ff$ in \eqref{ffiJnYn} belongs to $C_0(0,1]$ (and thus may
be a candidate for being an eigenfunction of $T$) if and only
if $c_2=0$.

\medskip\noindent
As a consequence of this and of Proposition \ref{p:fiq} we obtain the following.
\begin{corollary}\label{c:eigenvalue}
Consider the operator $T$ from \eqref{Tdef}, and let
$\EV>0$. Then $\EV>0$ is an eigenvalue of $T$ if and
only if
\begin{equation*}
J_m\left(2\sqrt{n/\EV}\right)=0.
\end{equation*}
In this case
\begin{equation*}
\ff(x)=x^{n/2} q'(x) =c_1 x^{1/2} J_n(2\sqrt{n/\EV} \sqrt{x}), \quad c_1\neq 0
\end{equation*}
are the only eigenfunctions corresponding to $\EV$.
\end{corollary}
For given $m$ let us denote the roots of $J_m$ by
$\alpha_{m,r}$ ($r\in\NN$) ordered increasingly. At this stage it is in
order to recall the following about the zeros of Bessel
functions. We have $\alpha_{m,r}\to \infty$ ($r\to\infty$), and
for rather large values the roots $\alpha_{m,r}$ of $J_m$ are
very well distributed, as essentially there falls one root in
each interval of length $\pi$. However, for fixed $m$ the
increasing sequence of zeros $(\alpha_{m,r})$ starts only with
$\alpha_{m,1} \sim m + c m^{1/3}$, with $c=1.8557571\dots$, see
\cite[9.5.14, p. 341]{AS}. Let us introduce the notation
\begin{align*}
 \EV_{m,r}&:=\EV_r:=4(m+1)/\alpha_{m,r}^2
 \intertext{and}
q_r(x)&:=q_{m,r}(x):=x^{-m/2} J_{m}\Bigl(2\sqrt{(m+1)/\EV_{m,r}}\cdot  \sqrt{x}\Bigr)=x^{-m/2} J_{m}\Bigl(\alpha_{m,r}\sqrt{x}\Bigr).
\end{align*}
By putting everything together we obtain the following result.
\begin{theorem}\label{t:summary}
For the extremal problem \eqref{Sproblem}
we have
\begin{equation}\label{eq:mainSk}
S(k)=\EV_1=\frac{4 (k-1)}{\alpha_{k-2,1}^2}=
\frac{4}{k + 2c k^{1/3} + O(1)}
\end{equation}
with $\alpha_{k-2,1}$ the first root of the order $k-2$ Bessel
function $J_{k-2}(x)$, and the constant $c=1.8557571\dots$.

The only extremal functions for the formulation
\eqref{qproblem} are nonzero constant multiples of
\begin{align*}
q_1(x)&=x^{-(k-2)/2}J_{k-2}(\alpha_{k-2,1}\sqrt{x}).
\end{align*}
\end{theorem}
\begin{proof}
Uniqueness follows from Remark \ref{uniqremark} and Corollary
\ref{c:eigenvalue}. For $q_1$ being a maximizer for
\eqref{qproblem} it remains only to show that
\begin{equation*}
\int_0^1 q_1(y)y^n dy\neq 0.
\end{equation*}
But this follows since $q_1$ is (strictly) positive all over
$(0,1)$, $\alpha_{k-2,1}$ being the very first zero of the
Bessel function $J_{k-2}$.
\end{proof}


\section{Power series}\label{sec:ps}

The above settles the issue of the best weight $P$---and also
the order of $S(k)$---in the GPY method. However, not all
weight functions are easy to handle, and a Bessel
function---even if an analytic function with relatively
strongly convergent series expansion---may be unmanageable, at
least in our current technical abilities. We will discuss,
using the classical series expansion of $J_m$, how it may work
in this context. Actually, not too well.


With the notations from the above for $q_r$, $\alpha_{m,r}$ etc.,
\eqref{Besselseries} yields
\begin{equation*}
q_{r}(x)=x^{-m/2} \beta^{-m} J_m(2\beta\sqrt{x}) = \sum_{j=0}^{\infty}
\frac{(-b)^j}{j!(m+j)!} x^{j}, \qquad \textrm{with} \quad \beta:=\sqrt{b},~ b=\beta^2 = \frac{\alpha_{m,r}^2}{4},
\end{equation*}
Then we can evaluate the functionals
\begin{align*}
G(q_r)&= \int_0^1 x^{m+1} q'^2_r(x)dx, \quad
F(q_r)= \int_0^1 (m+1)x^{m} q_r^2(x)dx =(-2) \int_0^1 x^{m+1} q_r(x) q'_r(x) dx.
\end{align*}
Computations with the power series provide
\begin{align*}
G(q_r)&= b^2 \sum_{j,\ell=0}^{\infty} \frac{(-b)^{j+\ell}}{(m+j+\ell+2)j!\ell ! (m+j+1)! (m+\ell+1)!},
\intertext{and}
F(q_r)&= 2 b \sum_{j,\ell=0}^{\infty} \frac{(-b)^{j+\ell}}{(m+j+\ell+2)j!\ell ! (m+j)! (m+\ell+1)!}~.
\end{align*}
With some reformulations we can also write
\begin{align*} 
&\frac{F(q_r)}{G(q_r)}
=\frac{2}{b} \frac{\sum_{\nu=0}^{\infty} \frac{(-b)^{\nu}}{(m+\nu)!(m+\nu+2)!} \sum_{j=0}^{\nu} \binom{m+\nu}{\nu-j} \binom{m+\nu+1}{j}} {\sum_{\nu=0}^{\infty} \frac{(-b)^{\nu}}{(m+\nu+1)!(m+\nu+2)!} \sum_{j=0}^{\nu} \binom{m+\nu+1}{\nu-j} \binom{m+\nu+1}{j}} \notag
\\ &= \frac{2}{b} \frac{\sum_{\nu=0}^{\infty} \frac{(-b)^{\nu}}{(m+\nu)!(m+\nu+2)!} \binom{2m+2\nu+1}{\nu}}{\sum_{\nu=0}^{\infty} \frac{(-b)^{\nu}}{(m+\nu+1)!(m+\nu+2)!} \binom{2m+2\nu+2}{\nu}}= \frac{1}{b} \frac{\sum_{\nu=0}^{\infty} \frac{(-b)^{\nu}(2m+\nu+2)}{(m+\nu+1)!(m+\nu+2)!}\binom{2m+2\nu+2}{\nu}}{\sum_{\nu=0}^{\infty} \frac{(-b)^{\nu}}{(m+\nu+1)!(m+\nu+2)!} \binom{2m+2\nu+2}{\nu}} \notag
\\ &= \frac{2}{b} \frac{\sum_{\nu=0}^{\infty} \frac{(-b)^{\nu}}{(2m+2\nu+2)!}\binom{2m+2\nu+2}{m+\nu}\binom{2m+2\nu+1}{\nu}}{\sum_{\nu=0}^{\infty} \frac{(-b)^{\nu}}{(2m+2\nu+3)!}\binom{2m+2\nu+3}{m+\nu+1} \binom{2m+2\nu+2}{\nu}}
= \frac{2}{b} \frac{\sum_{\nu=0}^{\infty}
\frac{(2m+2\nu+1)!}{(m+\nu)!(m+\nu+2)!(2m+\nu+1)!\nu!}(-b)^{\nu} }{\sum_{\nu=0}^{\infty} \frac{(2m+2\nu+2)!}{(m+\nu+1)!(m+\nu+2)!(2m+\nu+2)!\nu!}(-b)^{\nu}}.
\end{align*}

Unfortunately the series expansions here have large and
oscillating terms, so dealing with it does not seem to be
simple. When, e.g., $b$ is of the order $m^2$, then also the
terms with $\nu \approx m$ are the highest, and there are a
large number of similar order large terms. Therefore, this
series expansion does not seem to be suitable neither for the
computation of the value of the ratio, nor for the extraction
of a good polynomial approximation which would approach the
global maximum while remaining manageable.


\section{Approximate maximization by polynomials}\label{sec:poly}

\noindent\textbf{5.1.} In the aimed applications in showing small gaps between
consecutive prime numbers it is very important to have a
suitably nice, manageable function $P$. It is enough to mention
that even in the simplest case of $P(x)=x^{k+\ell}$, $\ell \asymp \sqrt{k}$ the technical difficulties become rather serious when $k$ and $\ell$ tend to infinity with the size $N$ of the primes, see \cite{GPYII}. The details of these aspects, when the choice of the weight function is done according to the present work, will be handled in the forthcoming paper \cite{Pintzuj}. Here we will only present the foreseen choice of the weight $P$, and show its approximate optimality. The said choice will be a relatively simple function, actually a real polynomial $P(x)$, satisfying the conditions
\begin{equation}\label{Polyproperties}
x^k | P(x), ~~ P(1)> 0, ~~ \deg P(x) = k +C_0 k^{1/3},
\end{equation}
which is essentially optimal in the extremal problem
\eqref{Sproblem}. More exactly, with the notations
\begin{align}\label{Ak}
A_k:=\int_0^1\frac{x^{k-2}}{(k-2)!} \left( P^{(k-1)}(1-x)\right)^2 dx,
\qquad
B_k:=\int_0^1\frac{x^{k-1}}{(k-1)!} \left( P^{(k)}(1-x)\right)^2 dx 
\end{align}
it satisfies with an absolute constant $C_1$
\begin{equation}\label{AkminusBk}
A_k(k+C_1k^{1/3}) - 4 B_k \geq 0.
\end{equation}
Equivalently,
\begin{equation}\label{SPk}
S(P,k):=\frac{A_k}{B_k} \geq \frac{4}{k+C_1 k^{1/3}},
\end{equation}
in full correspondence with  \eqref{eq:mainSk}.
Compared with the optimal transformed Bessel function $q_1(x)=
x^{-m/2}J_m(\alpha_{m,1}\sqrt{x})$, the difference is only in
the value of the constant $C_1$.

In order to define our polynomial we put
\begin{equation}\label{Mgdef}
M:=\lceil C_1 k^{1/3}/6 \rceil,\qquad g(y):=(y-1)^4(2-y)^4.
\end{equation}
Let us remark that the exact choice of $g(y)$ is irrelevant,
any positive polynomial or even a function $g\in C^1[1,2]$ with
a zero of order at least 3 at $y=1$ and $y=2$ would suffice for
our purposes. After this, let
\begin{equation}\label{Pkdef}
P(x):=P_k(x):=\sum_{\ell=M\atop 2\nmid \ell}^{2M} g\left(
\frac{\ell}{M}\right) \left( \frac{k}{2}\right)^{\ell}
\frac{x^{k+\ell}}{(k+\ell)!}~.
\end{equation}
\textbf{5.2.} In evaluating $A_k$ and $B_k$ we will use the well-known relation (easily obtained by partial integration and induction) for the Euler
integral
\begin{align}\label{EulerB}
B(m,n)&:=\int_0^1 x^{n} (1-x)^{m} dx = \frac{m!}{(n+1)\cdots(m+n)} \int_0^1 x^{m+n} dx = \frac{n!~m!}{(m+n+1)!}.
\end{align}
In view of \eqref{Pkdef}--\eqref{EulerB} we have
\begin{align}\label{Pformuli}
P^{(k-1)}(x)&=\sum_{\ell=M\atop 2\nmid \ell}^{2M} g\left( \frac{\ell}{M}\right) \left( \frac{k}{2}\right)^{\ell} \frac{x^{\ell+1}}{(\ell+1)!},~\\
P^{(k)}(x)&=\sum_{\ell=M\atop 2\nmid \ell}^{2M} g\left( \frac{\ell}{M}\right) \left( \frac{k}{2}\right)^{\ell} \frac{x^{\ell}}{\ell !}, \\
\left(P^{(k-1)}(x)\right)^2 &=\sum_{\ell_1=M\atop 2\nmid \ell_1}^{2M} \sum_{\ell_2=M\atop 2\nmid \ell_2}^{2M} g\left( \frac{\ell_1}{M}\right) g\left( \frac{\ell_2}{M}\right) \left( \frac{k}{2}\right)^{\ell_1+\ell_2} \frac{x^{\ell_1+\ell_2+2}}{(\ell_1+1)!(\ell_2+1)!}, \\
\left(P^{(k)}(x)\right)^2 &=\sum_{\ell_1=M\atop 2\nmid \ell_1}^{2M} \sum_{\ell_2=M\atop 2\nmid \ell_2}^{2M} g\left( \frac{\ell_1}{M}\right) g\left( \frac{\ell_2}{M}\right) \left( \frac{k}{2}\right)^{\ell_1+\ell_2} \frac{x^{\ell_1+\ell_2}}{\ell_1!~\ell_2!}, \\
A_k &=\sum_{\ell_1=M\atop 2\nmid \ell_1}^{2M} \sum_{\ell_2=M\atop 2\nmid \ell_2}^{2M} g\left( \frac{\ell_1}{M}\right) g\left( \frac{\ell_2}{M}\right) \frac{\left({k}/{2}\right)^{\ell_1+\ell_2}}{(k+\ell_1+\ell_2+1)!} \binom{\ell_1+\ell_2+2}{\ell_1+1}, \\
B_k &=\sum_{\ell_1=M\atop 2\nmid \ell_1}^{2M} \sum_{\ell_2=M\atop 2\nmid \ell_2}^{2M} g\left( \frac{\ell_1}{M}\right) g\left( \frac{\ell_2}{M}\right) \frac{\left({k}/{2}\right)^{\ell_1+\ell_2}}{(k+\ell_1+\ell_2)!} \binom{\ell_1+\ell_2}{\ell_1}.
\end{align}
In the following, put
\begin{align}\label{uvHD}
u:=\ell_1+1, \:v&:=\ell_2+1,\quad \textrm{and} \quad H:=(u+v)/2, \: D:=(u-v)/2.
\end{align}
Taking into account $2\nmid \ell_1, \ell_2$ and $g(1)=g(2)=0$,
the even variables $u,v$ will run from $M+1$ to $2M$ and we
will have
\begin{equation}\label{k!l!}
\frac{k^{\ell_1+\ell_2} k!}{(k+\ell_1+\ell_2)!} = 1 + O\left( \frac{M^2}{k}\right).
\end{equation}
Here and elsewhere in the sequel the implied absolute constants
of the $O$ symbol as well as the absolute constants $C_i$
($i\geq 2$) will be always independent from $C_1$.

Clearly, we can replace $A_k':=k! A_k$ and $B_k':=k! B_K$ for
$A_k$ and $B_k$, resp., in \eqref{AkminusBk}, hence in order to
show \eqref{AkminusBk} it suffices to prove
\begin{equation}\label{Ikgoalnonneg}
\sum_{H=M+1}^{2M} I_k(H) \left(1 + O\left( \frac{M^2}{k}\right)\right) \geq 0,
\end{equation}
where
\begin{equation}\label{IkHdef}
I_k(H):=2^{-2H} \sum_{2 | u,v, ~u+v=2H \atop u,v \in (M,2M]} g\left( \frac{u-1}{M}\right)
g\left( \frac{v-1}{M}\right) \binom{u+v-2}{u-1} \left\{ \frac{k+C_1 k^{1/3}}{k+4M}
\frac{(u+v)(u+v-1)}{uv}-4 \right\}.
\end{equation}
Let us denote the above summation conditions simply by
$\sum^*_{u,v}$ and let us consider first
\begin{align}\label{Ik'}
I_k'(H)&:=2^{-2H} \sum\nolimits^*_{u,v} g\left( \frac{u-1}{M}\right) g\left( \frac{v-1}{M} \right)
\binom{u+v-2}{u-1} \left\{\frac{(u+v)(u+v-1)}{uv}-4 \right\}
\notag \\ &=2^{-2H} \sum\nolimits^*_{u,v} g\left( \frac{u-1}{M}\right) g\left( \frac{v-1}{M}\right)
\binom{u+v-2}{u-1} ~\frac{4D^2-2H}{H^2-D^2} .
\end{align}
We will see later that the mere significance of the weight
function $g$ is to cut the tails when $H$ is near to $M$ or
$2M$, so first we will investigate the simpler sum
\begin{equation}\label{Ik"}
I_k''(H):=2^{-2H} \sum\nolimits^*_{u,v} \binom{u+v-2}{u-1} ~\frac{4D^2-2H}{H^2-D^2} ,
\end{equation}
when
\begin{equation}\label{HtM}
H\in [M+t(M),2M-t(M)],\quad t(M):=4\sqrt{M\log M}.
\end{equation}
By Stirling's formula
\begin{equation}\label{Stirling}
\log \Gamma(s)=\bigl(s-\tfrac12\bigr)\log s - s + \tfrac12 \log\bigl(2\pi\bigr) +
O(\tfrac1{|s|}),
\end{equation}
we obtain
\begin{align}\label{logsummand}
\log  & \left(2^{-2H}  \binom{u+v-2}{u-1} \right) = \log \left( \frac{\Gamma(2H-1)
2^{-2H}}{\Gamma(H+D)\Gamma(H-D)} \right)
\notag \\ &  = \left(2H-\frac32 \right) \left( \log H + \log 2 + \log
\left( 1-\frac1{2H}\right)\right) +1-\frac{\log2\pi}{2}-2H\log 2 + O\left( \frac{1}{H}\right)
\notag \\ &\qquad -\left( H+D-\frac12 \right) \left(\log H+\log \left( 1+\frac{D}{H}\right)\right)
- \left( H-D-\frac12 \right) \left( \log H + \log \left( 1-\frac{D}{H}\right)\right)
\notag \\
& = \frac{-\log H}{2} - \frac{3\log 2}{2} - 1 + 1 - \frac{\log 2\pi}{2}
\notag \\ &\qquad - \sum_{n=0}^{\infty} \left( \frac{2D}{2n+1} \left(\frac{D}{H}\right)^{2n+1} - \frac{2H}{2n+2} \left(\frac{D}{H}\right)^{2n+2} \right)+ O\left( \frac{1}{H} + \frac{D^2}{H^2}\right).
\notag \\
& = \frac{-\log H}{2} - \frac{\log 16 \pi}{2} - \sum_{n=0}^{\infty} \frac{1}{(n+1)(2n+1)} \frac{D^{2n+2}}{H^{2n+1}} + O\left( \frac{1}{H} + \frac{D^2}{H^2}\right).
\end{align}
Using $e^{-|x|}=1+O(x)$ we obtain from \eqref{Ik"}, \eqref{HtM}
and \eqref{logsummand}
\begin{align}\label{Ik"cutesti}
I_k''(H)&=\frac{1}{2\sqrt{\pi}H^{3/2}}  \hskip-0.5em\sum_{2|D- H \atop {M<H-D,H+D\leq 2M}} \hskip-0.5em
e^{-D^2/H} \left( 2\frac{D^2}{H} -1\right) \left( 1+O\left( \frac{1}{H}+\frac{D^2}{H^2}
+ \frac{D^4}{H^3}\right) \right) \notag \\
&=\frac{I_k^*(H)}{2\sqrt{\pi}H^{3/2}}  + O\left( \frac{1}{M^2}\right)
+O\biggl( \int_{t(M)}^\infty \frac{t^2}{H}e^{-t^2/H} dt\biggr)
= \frac{I_k^*(H)}{2\sqrt{\pi}H^{3/2}}  + O\left( \frac{1}{M^2}\right),
\end{align}
where
\begin{equation}\label{Ikstar}
I_k^*(H):= \sum_{2|m- H  \atop m=-\infty}^\infty f(m),\qquad f(x)=\left( 2\frac{x^2}{H}-1\right)e^{-x^2/H}.
\end{equation}

\medskip\noindent\textbf{5.3.}
We will use Poisson summation formula
\begin{equation}\label{Poisson}
\sum_{n=-\infty}^{\infty} f(t+nT) = \frac{1}{T} \sum_{\nu=-\infty}^{\infty}\widehat{f}\left( \frac{\nu}{T}\right) e^{2\pi i \nu t/T}
\end{equation}
with $T=2$, $t=0$ for $2|H$ and $T=2$, $t=1$ for $2\nmid H$, valid if\footnote{For this form see \cite[Chapter VII, \S 2]{SW}.}
\begin{equation}\label{fplusfhat}
|f(x)|+|\widehat{f}(x)|\leq C(1+|x|^{-1-\de}) \quad \textrm{with some}\quad  \de>0, C>0.
\end{equation}

Let us consider first the case when $H$ is even, that is when $m=2n$ in the above sum \eqref{Ikstar}. We have then
\begin{align}\label{I}
2 I_k^* (H)&= 2\sum_{n=-\infty}^{\infty} f(2n) = \sum_{\nu=-\infty}^{\infty} \widehat{f}(\nu/2)
=\sum_{\nu=-\infty}^{\infty} ~ \int_{-\infty}^{\infty} f(x) e^{\pi i \nu x} dx
\notag \\ &=\sum_{\nu=-\infty}^{\infty} ~ \int_{-\infty}^{\infty} e^{-x^2/H} \left(2\frac{x^2}{H} -1 \right)e^{\pi i \nu x} dx
= \sqrt{H} \sum_{\nu=-\infty}^{\infty} ~ \int_{-\infty}^{\infty}
e^{-y^2} \left(2y^2-1\right) e^{\pi i \nu \sqrt{H} y} dy \notag \\
&=\sqrt{H} \sum_{\nu=-\infty}^{\infty} ~\int_{\Im y=0} e^{-(y-\pi i
\nu \sqrt{H}/2)^2-\pi^2\nu^2H/4} \notag \\ &\qquad \cdot \left\{
2(y-\pi i \nu \sqrt{H} /2)^2 + 2 (y-\pi i \nu \sqrt{H} /2) \pi i \nu
\sqrt{H} -\pi^2\nu^2H/2-1\right\} dy \notag \\ & = \sqrt{H}
\sum_{\nu=-\infty}^{\infty} \int_{\Im z=-\pi \nu \sqrt{H}/2}
e^{-z^2-\pi^2\nu^2H/4} \left\{ 2z^2+ 2 \pi i \nu \sqrt{H} z -
\pi^2\nu^2 H /2 -1 \right\} dz \notag \\ & = \sqrt{H}
\sum_{\nu=-\infty}^{\infty} ~ \int_{\Im z=0} e^{-z^2-\pi^2\nu^2H/4}
\left\{ 2z^2+ 2 \pi i \nu \sqrt{H} z - \pi^2\nu^2 H /2 -1 \right\}
dz \notag \\ & =\sqrt{H} \int_{z\in\RR} e^{-z^2} (2z^2-1) dz +
O\left( e^{-2H}\right) =
\sqrt{H}\left[-ze^{-z^2}\right]_{-\infty}^{\infty} +O\left(
e^{-2H}\right)=O\left( e^{-2H}\right).
\end{align}
The proof
runs completely analogously for $t=1$, that is, when $H$ is
odd, because the extra factor $e^{\pi i \nu}$ does not change
the modulus of $\widehat{f}(\nu/2)$.

\medskip\noindent\textbf{5.4.} Dealing with the original integral $I_k'(H)$ we have to take into account the effect of the weight function $g$ as well. From the Taylor expansion of $g$ we find
\begin{align}\label{gTaylor}
g\left( \frac{H\pm D-1}{M} \right) & = g\left( \frac{H-1}{M}\right)
\pm g'\left( \frac{H-1}{M}\right) \frac{D}{M} + O \left( \frac{D^2}{M^2}\right)
\end{align}
so this effect is
\begin{equation}\label{geffect}
g\left( \frac{H+D-1}{M} \right) g\left( \frac{H-D-1}{M} \right) = g^2 \left( \frac{H-1}{M}\right) + O \left( \frac{D^2}{M^2}\right).
\end{equation}
The effect of the error term on $I_k'(H)$ is here, similarly to
\eqref{Ik"},
\begin{equation}\label{OH32}
O \left( H^{-3/2} \sum_{2|D-H \atop |D|<M} e^{-D^2/H} \left(1+\frac{D^2}{H} \right)
\left(1+\frac{D^4}{H^3} \right) \frac{D^2}{M^2}\right) = O \left( \frac{1}{M^2}\right).
\end{equation}
So we obtain for all integers $H$ subject to \eqref{HtM} from
\eqref{Ik'}, \eqref{Ik"}, \eqref{Ik"cutesti}, \eqref{Poisson},
\eqref{I}, \eqref{geffect} and \eqref{OH32}
\begin{equation}\label{OIkprime}
I_k'(H)=I''_k(H) g^2 \left( \frac{H-1}{M}\right) + O \left(
\frac{1}{M^2}\right) = \frac{g^2 \left(
\frac{H-1}{M}\right)}{2\sqrt{\pi} H^{3/2}} I^*_k(H) + O \left(
\frac{1}{M^2}\right) \geq - \frac{C_2}{M^2}.
\end{equation}
On the other hand, if \eqref{HtM} does not hold, that is if
\begin{equation}\label{HMt}
\min \left( H-M,2M-H\right) < t(M),
\end{equation}
then we find $|D| < t(M)$ and $|H\pm D-1| \in [M,M+2t(M)]\cup
[2M-2t(M),2M]$, whence
\begin{equation}\label{DtM}
g\left( \frac{H+D-1}{M} \right) g\left(\frac{H-D-1}{M} \right)
\ll \frac{t^4(M)}{M^4} \ll \frac{1}{M^{7/4}}.
\end{equation}
This implies in case of \eqref{HMt}
\begin{align}\label{IkprimeHMt}
I'_k(H) \ll M^{-7/4} M^{-3/2} \int_{-\infty}^{\infty} e^{-u^2/H} \left( \frac{u^2}{H} +1\right)^3 du \ll M^{-11/4}.
\end{align}
Consequently, for $k>k_0$ we have by $I_k(H)\geq I_k'(H)$
\begin{equation}\label{sumIk}
\sum_{H \atop \min(H-M,2M-H)<t(M)} \hskip-0.5em I_k(H) \geq  \hskip-0.5em\sum_{H \atop \min(H-M,2M-H)<t(M)}  \hskip-0.5emI'_k(H) \geq -C_3 M^{-2}.
\end{equation}
Finally, let us investigate now for $H$ in \eqref{HtM} the difference
\begin{equation}\label{}
I_k(H)  - I'_k(H) \geq \frac{M}{k} \cdot 3 \widetilde{I_k}(H),
\end{equation}
where similarly to the above considerations
\begin{align}\label{Iktilde}
\widetilde{I_k}(H) &:= 2^{-2H} \sum\nolimits^*_{u,v} g\left( \frac{u-1}{M} \right)
g\left( \frac{v-1}{M}\right) \binom{u+v-2}{u-1}
\notag \\ &= \frac{1}{2\sqrt{\pi H}}  \hskip-0.5em\sum_{2|D- H \atop M<H-D, H+D\leq 2M} \hskip-0.5em e^{-D^2/H}
\left( 1 + O\left(\frac1{H} + \frac{D^2}{H^2} + \frac{D^4}{H^3} \right) \right)
\\ \notag & \qquad \qquad \qquad \qquad \cdot \left( g^2 \left( \frac{H-1}{M}\right)
+ O \left( \frac{D^2}{M^2}\right) \right) \geq C_4 g^2 \left(
\frac{H-1}{M}\right) - \frac{C_5}{M}.
\end{align}
Summing over all $H$ in \eqref{HMt} we obtain
\begin{equation}\label{edgesum}
\sum_{H=M+t(M)}^{2M-t(M)} I_k(H) \geq \sum_{H=M+t(M)}^{2M-t(M)}
I'_k(H) +\frac{C_6 M^2}{k}.
\end{equation}
Therefore by \eqref{OIkprime} and \eqref{sumIk}--\eqref{edgesum} we have finally
\begin{equation}\label{finalsumIk}
\sum_{H=M+1}^{2M} I_k(H) \geq \frac{C_6 M^2}{k} - \frac{C_2}{M} -
\frac{C_3}{M^2} > 0,
\end{equation}
if $C_1$ was chosen sufficiently large, satisfying
\begin{equation}\label{C1condi}
C_1> 6 \left( \frac{C_2}{C_6}\right) ^{1/3}.
\end{equation}

\bibliographystyle{degruyter-plain}
\bibliography{ThreeQuarter}

\parindent=0pt
\end{document}